\documentclass[a4paper,11pt]{article}
\usepackage{fullpage}
\usepackage{mathtools}

\usepackage{amssymb, amsthm, amsmath}
\PassOptionsToPackage{hyphens}{url}\usepackage[backref,colorlinks,citecolor=blue,bookmarks=true]{hyperref}

\newtheorem{main}{Theorem}
\newtheorem*{theorem*}{Theorem}
\newtheorem{theorem}{Theorem}[section]
\newtheorem{lemma}[theorem]{Lemma}
\newtheorem{proposition}[theorem]{Proposition}

\newtheorem{observation}[theorem]{Observation}
\newtheorem{fact}[theorem]{Fact}
\newtheorem{corollary}[theorem]{Corollary}

\theoremstyle{definition}
\newtheorem{definition}[theorem]{Definition}
\newtheorem{example}[theorem]{Example}
\newtheorem{remark}[theorem]{Remark}

\newtheorem*{manualthm}{Theorem}

\newtheorem*{manuallemma}{Lemma}

\newcommand{\F}{\mathbb{F}}

\newcommand{\CC}{\mathbb{C}}
\newcommand{\sub}{\subseteq}

\newcommand{\setmid}{\,\vert\,}

\renewcommand{\a}{\alpha}
\renewcommand{\b}{\beta}
\newcommand{\g}{\gamma}

\renewcommand{\d}{\delta}
\newcommand{\D}{\Delta}

\DeclareMathOperator{\Image}{Im}
\DeclareMathOperator*{\Exp}{\mathbb{E}}

\newcommand{\x}{\vec{x}}

\DeclareMathOperator{\PR}{PR}

\DeclareMathOperator{\AR}{AR}

\DeclareMathOperator{\rk}{rk}
\DeclareMathOperator{\bias}{bias}

\DeclareMathOperator{\ch}{ch}
\DeclareMathOperator{\Poly}{Poly}
\DeclareMathOperator{\Forms}{Form}

\DeclareMathOperator{\codim}{codim}
\DeclareMathOperator{\cor}{cor}
\DeclareMathOperator{\rank}{rank}

\newcommand{\polar}[1]{\overline{#1}}

\renewcommand{\P}{\mathbb{P}}

\DeclareMathOperator{\Z}{Z}

\newcommand{\Nabla}{\boldsymbol{\nabla}}

\newcommand{\vvert}{\,\vert\,}

\DeclareMathOperator{\GI}{GI}

\newcommand{\inv}[1]{#1^{\raisebox{0.2ex}{$\scriptscriptstyle\!-\!1$}}}

\def\dj{d\kern-0.4em\char"16\kern-0.1em}

\renewcommand{\vec}[1]{\mathbf{#1}}

\renewcommand{\v}{\vec{v}}
\renewcommand{\x}{\vec{x}}

\newcommand{\norm}[1]{\lVert{#1}\rVert}

\newcommand{\Ot}{\tilde{O}}

\newcommand{\1}{\mathbf{1}}

\DeclareMathOperator{\bi}{\mathfrak{bias}} 

\title{Nearly-polynomial inverse theorem for the $U^d$ norm in degree $d+1$}

\author{
	Tomer Milo\thanks{
        School of Mathematical Sciences, Tel Aviv University,
        Ramat Aviv, 
        Tel Aviv 69978, Israel.
        Email: \texttt{tomermilo@mail.tau.ac.il}.
        Work on this project began as part of a summer graduate research assistant position, supported by NSF Award DMS-2302988.}
    \and
    Guy Moshkovitz\thanks{
        Department of Mathematics, City University of New York (Baruch College \& Graduate Center), 
        New York,
        NY 10010, USA. 
        Email: \texttt{guymoshkov@gmail.com}.
        Supported by NSF Award DMS-2302988 \& PSC-CUNY award.}}
\date{}

\begin{document}
\maketitle



\begin{abstract}
    We prove a nearly polynomial inverse theorem for the Gowers $U^d$ norm, over finite fields of non-small characteristic, for polynomials of degree $d+1$. The case of degree $d$ was very recently settled 
    by Mili\'{c}evi\'{c} and Ran\dj elovi\'{c}
    with a fully polynomial bound. We moreover provide a nearly polynomial inverse theorem for \emph{homogeneous} polynomials of any degree smaller than $2d$.

    Our methods may be of independent interest, and include a refined notion of polynomial decomposition that captures correlation with  polynomials of lower degree than classical notions do, and a new correlation lemma that improves upon similar lemmas in the literature.
    Additionally, we illustrate the usefulness of the new correlation lemma by using it to give an alternative proof for the aforementioned result of  Mili\'{c}evi\'{c} and Ran\dj elovi\'{c}.
\end{abstract}

\section{Introduction}

Let $\F$ be a finite field, for simplicity prime $\F=\F_p$.
The \emph{discrete derivative} of a function $f \colon \F^n \to \F$ in direction $v \in \F^n$ at the point $x \in \F^n$ is $\D_v f(x) = f(x+v)-f(x)$.
More generally, the (discrete) derivative of order $d$,
in directions $\vec{v}=(v^{(1)},\ldots,v^{(d)}) \in (\F^n)^d$ at $x \in \F^n$, is
\[\D^d_{\v} f(x) := \D_{v^{(d)}}\cdots\D_{v^{(1)}} f(x).\]
We write $\D^d f$ for the map $(x,\v) \mapsto \D^d_{\v}f(x)$.
A function $f \colon \F^n \to \F$ can be written\footnote{Or: the unique reduced polynomial (variable degrees below $|\F|$) agreeing with $f$ has (total) degree below $d$.}
as a polynomial of degree below $d$ if and only if $\D^d f = 0$ (i.e., $\D^d_\v f(x) = 0$ for every $(x,\v) \in (\F^n)^{d+1}$).



An important (qualitative) result about polynomials over $\F$, of characteristic $\ch(\F)>d$, is that the derivative characterization is robust: 
$f$ is \emph{correlated} with some polynomial of degree below $d$
if and only if $\D^d f$ is \emph{biased}.
This lies at the core of the theory of Gowers norms~\cite{Gow01,Tao12,HHL19}, 
and is usually formulated using exponential sums:\footnote{Which are closely related to the distance of a distribution from uniformity; see e.g.\ Claim~33 in~\cite{BV10}.} 
for any function $f \colon \F^m \to \F$, 
one denotes 
\[\bias(f) = |\Exp_{x \in \F^m} e(f(x))| \quad\text{and}\quad \cor(f,P) = \bias(f-P)\]
where $e(y):=e^{2\pi i y/p} \colon \F\to\mathbb{C}$ and $\Exp_{x \in A} := \frac{1}{|A|}\sum_{x \in A}$.
We note that $\bias(f) \in [0,1]$,\footnote{For example, if $\deg(f)<d$ then $\bias(\D^d f)=1$, while if $f$ is random then $\bias(\D^d f)=o(1)$ as $n\to\infty$.}
and that, explicitly, $\bias(\D^d f)=\Exp_{x \in \F^n,\vec{v} \in (\F^n)^d} e(\D^d_{\vec{v}} f(x))$.\footnote{The absolute value is not needed for $\D^d f$.}
The forward direction of this characterization---correlation implies bias---is not hard to prove, even quantitatively (e.g., using\footnote{Indeed, if $\deg(P)<d$ then
$\bias(\D^d f) = \bias(\D^d(f-P)) \ge |\bias(f-P)|^{2^d} = \cor(f,P)^{2^d}$.} the Cauchy-Schwarz-Gowers inequality or the monotonicity of Gowers norms~\cite{Tao12}):
\begin{equation}\label{eq:GI-easy}
    \bias(\D^d f) \ge \cor_{<d}(f)^{2^d}
\end{equation}
where $\cor_{<d}(f) := \max_{\deg(P)<d} \cor(f,P)$.

The inverse direction is much harder.
For $d=2$, the question falls in the domain of Fourier analysis,
and it follows using Parseval's identity that $\cor_{<2}(f) \ge \bias(\D^2 f)^{1/2}$.\footnote{Indeed, the magnitudes of the Fourier coefficients of $g=e(f)$ are $\cor(f,\ell)$, for any linear form $\ell$ (or character), and $\bias(\D^2 f)=\sum_\g |\hat{g}(\g)|^4 
= \sum_\ell \cor(f,\ell)^4
\le \cor_{<2}(f)^2 \sum_\ell \cor(f,\ell)^2 = \cor_{<2}(f)^2$.}
For general $d$, however, the inverse question falls in the domain of higher-order Fourier analysis,
and is only known to hold with a very weak bound~\cite{GM20} (stronger bounds are known in the integer setting~\cite{LSS24}).
The inverse question is often stated 
in terms of the Gowers (uniformity) norms of a complex-valued function.
%
For $g \colon \F^n \to \mathbb{C}$, 
the $U^d$-norm is
$\norm{g}_{U^d} = (\Exp_{(x,\v) \in (\F^n)^{d+1}} {\D^*}_{v^{(d)}}\cdots{\D^*}_{v^{(1)}} g(x))^{1/2^d}$,
with $\D^*_v g(x) = g(x+v)\overline{g(x)}$,
and the weak Gowers norm is
$\norm{g}_{u^d} = \max_{\deg(P)<d}\big|\Exp_{x \in \F^n} g(x)e(-P(x))\big|$.


\begin{theorem}[Gowers inverse over finite fields~\cite{BTZ10,TZ10,TZ12}]\label{thm:GI}
    For any $g \colon \F^n \to \mathbb{C}$ with $\|g\|_{\infty} \le 1$, 
    \[\norm{g}_{u^d} \ge \Phi(\|g\|_{U^d})\]
    for some function $\Phi=\Phi_{d,\F}>0$.
\end{theorem}
In the representative case where $g=e(f)$ with $f \colon \F^n\to\F$ (e.g., Lemma~1.2 in~\cite{TZ10}), 
the Gowers norm is simply $\norm{g}_{U^d} = \bias(\D^d f)^{1/2^d}$,
so Theorem~\ref{thm:GI} is a qualitative inverse to~(\ref{eq:GI-easy}):
\begin{equation}\label{eq:GI}
    \exists \Phi>0 \colon \cor_{<d}(f) \ge \Phi(\bias(\D^d f)).
\end{equation}
We note that if the field has small characteristic $\ch(\F)<d-1$, the maximum in Theorem~\ref{thm:GI}, over $P$ with $\deg(P) < d$,
must also range over so-called \emph{non-classical} polynomials of degree below $d$ (see~\cite{TZ12} and~\cite{BSST22,GT09,LMS11}), which are functions $\F^n\to\mathbb{R}/\mathbb{Z}$ that involve division by powers of $\ch(\F)$. 
In this paper we restrict ourselves to the high characteristic case (and so, to classical polynomials).


The Gowers inverse question goes back to foundational work of Gowers~\cite{Gow01}, which led to the development of higher-order Fourier analysis, and seminal works of Green and Tao on additive patterns in the primes~\cite{GT10} and the $U^3$-norm~\cite{GT08}, and of Samorodnitsky~\cite{Sam07} on polynomiality testing.
It has found applications in a large number of areas (see, e.g.,~\cite{BV10,HHL19,Pel24,Tao12,TZ10,Wol15} and the references within).
However, the lack of strong bounds for Theorem~\ref{thm:GI} implies the same for the applications.

The \emph{Polynomial Gowers Inverse} conjecture~\cite{GT08}
posits that, in fact, the relationship between bias and correlation is about as tight as it could be. Namely, the conjecture is that $\Phi(x)$ in Theorem~\ref{thm:GI} can be as large as a power of $x$ for any fixed $d$: $\Phi(x) \ge x^{O_d(1)}$ (this matches the easy direction $\Phi(x) \le x^{1/2^{d}}$ in~(\ref{eq:GI-easy}) up to the constant). 
It is perhaps the central conjecture of higher-order Fourier analysis.

While a recent breakthrough in~\cite{GGMT24,GGMT25} was achieved in the case $d=3$, the Polynomial Gowers Inverse conjecture for any $d>3$ remains wide open.

\subsection{Gowers inverse in low degree}

Let $\GI_d(k)$ denote
the special case of the inverse theorem for the Gowers $U^d$-norm~(\ref{eq:GI}) where 
the function $f \colon \F^n \to \F$ is a polynomial of degree at most $k$, and $\ch(\F)>k$.
Green and Tao~\cite{GT09} proved $\GI_d(k)$, in what was perhaps the first substantial result towards the Gowers inverse theorem.
Their proof uses a regularity lemma for polynomials, applied inductively on $k$, 
which results in an extremely weak, Ackermann-type bound.
In particular, for $\GI_d(d+1)$, their proof yields a bound that is at best exponential: $\Phi(x) = \exp(-(1/x)^{O_d(1)})$.\footnote{Even if one were to use an optimal bound in the induction basis $k=d$. This stems from the fact that the regularization step in Lemma~2.4 of~\cite{GT09}, even when applied with a linear growth function $F$, yields a partition with an exponential number of parts. }
Very recently in~\cite{MR25}, $\GI_d(d)$ was proved with a polynomial bound, the first\footnote{$\GI_d(k)$ is trivially true for $k<d$, since $\deg(f)<d$ implies $\cor_{<d}(f)=1$.} general case of the \emph{Polynomial} Gowers Inverse conjecture.
This result relates to a separate line of work on the partition-versus-analytic rank conjecture for multilinear forms (e.g.~\cite{Mil19,Jan19,CM22,AKZ21,CM23,MZ24,CY25}). Later, we explain in detail that connection, and how it relates to our work. 

In this paper we obtain an almost-polynomial bound for $\GI_d(d+1)$, the next open case.
With a slight abuse of notation, given $x \in (0,1]$ we write $x^{\Ot_d(1)}$ for 
$\exp(-O_d(\log(1/x)^{1+o(1)}))$.\footnote{For completeness, set $0^{\Ot_d(1)}=0$.}
(This is even stronger than a quasi-polynomial bound,
which has a constant $O_d(1)$ at the top instead of $1+o(1)$.)\footnote{
More precisely, we get $\log(1/x)\cdot\log(\log(1/x)+1)$ in the exponent.
(We even get a bound that improves as $|\F|$ increases, of the shape $x^{O_d(\log_{|\F|}(\log_{|\F|}(1/x)))}$, which for $|\F| \gg \log(1/x)$ is truly polynomial $x^{O_d(1)}$; see~\cite{MZ24}.)}

\begin{main}\label{thm:d+1-PGI-poly}
    Let $d \ge 1$.
    For every $f \colon \F^n \to \F$ with $\deg(f) \le d+1$, where $d+1 < \ch(\F)$,
    \[\cor_{<d}(f) \ge \bias(\D^d f)^{\Ot_d(1)}.\]
    
\end{main}
Put differently, $\norm{g}_{u^d} \ge \|g\|_{U^d}^{\Ot_d(1)}$
for any polynomial phase $g\colon \F^n\to\CC$ of degree at most $d+1$ (that is, $g=e(f) $ with $\deg(f) \le d+1$).

A main difficulty in the proof of Theorem~\ref{thm:d+1-PGI-poly} stems from the non-homogeneity of $f$. 
Standard proofs of the degree-$d$ case utilize the fact that if $\deg(f)\le d$, 
then the $d$-times derivative $\D^d f$ is a multilinear form, a class of polynomials for which we have good understanding of the interplay between structure and randomness. Concretely, we know how to convert high $\bias(\D^d f)$ into a decomposition of $f_d$---the top-degree homogeneous component of $f$---in terms of few lower-degree polynomials.
A standard pigeonhole argument is then used to obtain a significant correlation for $f_d$ with a polynomial $P$ of lower degree, $\deg(P) \le d-1$,
and thus for $f$ itself by absorbing its lower-degree components into $P$.

However, when $\deg(f)=d+1$ and $f$ is not homogeneous, one obtains from the multilinear argument (via $\D^{d+1} f$, which is also biased) only a decomposition of the top-degree component $f_{d+1}$ of $f$, so absorbing the rest of $f$---which is of degree $d$---into a correlating polynomial $P$ with $\deg(P) \le d-1$
is no longer an option.
In fact, it is entirely possible for the degree-$d$ component $f_d$ of $f$ to not have any short decompositions, even if $\bias(\D^d f)$ is large (e.g., $f=(x_1+1)h$ for a random $h$, see Remark~\ref{remark:D-bias-nonhom}).
Furthermore, even if one were to successfully obtain a short decomposition of $f$ itself in terms of polynomials of degrees below $\deg(f)$,
it is not clear how to use this to obtain a correlation with a polynomial of degree as small as $d-1=\deg(f)-2$.\footnote{Moreover, unlike in the degree-$d$ case, one cannot guarantee 
a short decomposition with homogeneous factors. In particular, the degree-$1$ factors in a decomposition of $f$ might not have a common zero (as they need not be homogeneous; e.g., $x_1$ and $x_1+1$), and so there might not even be a large (affine) subspace where the degree-$1$ factors can be discarded (which would have at least meant that all factors have degree at most $\deg(f)-2$ there).}

In the homogeneous setting, 
we can do more. Namely, we obtain almost-polynomial bounds 
for homogeneous polynomials of degree smaller than $2d$.

\begin{main}\label{thm: homogeneous GI 2d}
    Let $d \ge 1$.
    For homogeneous $f \colon \F^n \to \F$ with $\deg(f) < 2d$, 
    where $\deg(f) < \ch(\F)$, 
    \[ \cor_{<d}(f) \ge \bias(\D^d f)^{\Ot_d(1)} .\]
\end{main}

The proof of Theorem~\ref{thm: homogeneous GI 2d} relies on a new, general correlation lemma. 
Standard correlation lemmas/arguments in the literature (e.g.~\cite{Jan19}, proof of Theorem~1.6; or~\cite{GT09}, beginning of the proof of Theorem~1.4; or~\cite{KL08}, proof of Theorem 3; 
or~\cite{TZ12}, beginning of the proof of Theorem~1.10 from~Theorem 1.18; and many more)
use a simple pigeonhole-type argument on Fourier coefficients to obtain significant correlation with a lower degree polynomial, 
provided we have a full decomposition of the given polynomial.
Our new correlation lemma, on the other hand, requires a more relaxed condition, closely related to membership in an appropriate ideal.


Furthermore, as another application of our new correlation lemma, we show in Section~\ref{subsection: degree d} how it can be used to provide an alternative proof of a polynomial bound for $\GI_d(d)$
from a weak version of the partition-versus-analytic rank dichotomy~\cite{Mil25}.

\section{Preliminaries and proof outline}

\subsection{Preliminaries}

From now on $\F$ is an arbitrary finite field.
We write $\bi$ (in Fraktur letters) for the bias as a complex number:
for any function $f \colon \F^n \to \F$ over a finite field $\F$, denote
\[\bi(f) = \Exp_{x \in \F^n} \chi(f(x)) \in \mathbb{C} \quad\text{ and }\quad \bias(f)=|\bi(f)| \in [0,1]\]
where $\chi$ is any nontrivial additive character of $\F$ that we henceforth fix (for example, for $\F=\F_p$ prime, we can take $\chi(y)=e^{2\pi i y/p}$). 
Note $\cor(f,P) = \bias(f-P) \in [0,1]$.
We use bold letters, such as $\vec{A}=(A_1,\ldots,A_m)$, to emphasize that the object in question is a vector/tuple.
We denote the indicator function for a set $S$ by $\1_S(x) = \begin{cases}
    1 & x \in S\\
    0 & x \notin S
\end{cases}$.

We denote by $\Poly_k(\F)$ the set of (formal) polynomials of degree at most $k$, over the field $\F$, in any (finite) number of variables $x_1,\ldots,x_n$.\footnote{When combining polynomials, they are understood to be defined on the same variable set (adding dummy variables if needed).}
A \emph{form} is a homogeneous polynomial. We denote by $\Forms_k(\F)$ the set of forms of degree $k$ ($\ge 0$) over $\F$ in any number of variables. We simply write $\Forms_k$ when $\F$ is immaterial, or $\Forms$ when $k$ is immaterial as well.
Note that any $f \in \Poly_k$ is the sum of forms, namely, the homogeneous components of $f$ of each degree: $f = \sum_{i=0}^k f_i$ with $f_i \in \Forms_i$. We also denote $f_{<d} := \sum_{i=0}^{d-1} f_i$.

The \emph{rank} of a form (also called strength or Schmidt rank), denoted by $\rk(f)$, is the minimal $r$ for which $f$ is the sum of reducible forms:
\[ f= \sum_{i=1}^r g_ih_i \]
for some forms $g_i$ and $h_i$ of (positive) degree smaller than $\deg(f)$.
A form is \emph{multilinear} if it is linear in each of its variable vectors; more concretely, a multilinear form of degree $d$ ($\ge 1$), or a \emph{$d$-linear form} for short, is of the form $T=T(x^{(1)},\ldots,x^{(d)}): (\F^n)^d \to \F$
and is linear in each of $x^{(1)},\ldots,x^{(d)}$.
Equivalently, any monomial in the support of $T$ has one variable from each of $x^{(1)},\ldots,x^{(d)}$.
The \emph{partition rank}~\cite{Nas20}
of a $d$-linear form $T$, denoted by $\PR(T)$, is the minimal $r$ for which $T$ is the sum of reducible $d$-linear forms: 
\[ T = \sum_{i=1}^r R_iQ_i \]
for some multilinear forms $R_i = R_i(x_{I_i})$ and $Q_i=Q_i(x_{I_i^{c}})$
with $\emptyset \subsetneq I_i \subsetneq [d]$,
where  for $x \in (\F^n)^d$ we write $x_{I}$ for the restriction of $x$ to the coordinates corresponding to $I \sub [d]$. 
A related notion to the partition rank is that of the \emph{analytic rank}~\cite{GW11} over a finite field $\F$, denoted by $\AR(T)$:
\[ \AR(T) := -\log_{|\F|}(\bias(T)).\]
We note that for $T$ multilinear, $\bias(T)$ also has an elementary  definition: $\bias(T) = \P(T=0)-\P(T=c)$ ($\ge 0$), for any choice of nonzero $c \in \F$.\footnote{Indeed, 
multilinearity implies that $\P(T=t)$ is the same for every $t \in \F^\times$, so
$\bi(T) = \sum_{t \in \F} \P(T=t)\chi(t)
= \P(T=0)+\P(T=c)(\sum_{t \in \F}\chi(t)-\chi(0)) = \P(T=0)-\P(T=c)$.}
Moreover, $\AR(T) \le \PR(T)$ (e.g.,~\cite{Lov19}).

A fundamental part of our argument relies on the best bound known for the partition rank of biased multilinear forms.\footnote{For completeness, note that if $d=1$ then $\AR(T)=\PR(T)=\infty$.}
We write $\Ot_d(x) := O_d(x^{1+o(1)})$.\footnote{More precisely, we can get $\Ot_d(x)=O_d(x\log(x+1))$. (In fact, we can even get
$\Ot_d(x) = O_d(x(\log_{|\F|}(x+1)+1))$, which in particular is linear for mildly large $\F$).}

\begin{theorem}[\cite{MZ24}]\label{thm:SvR}
    For every $d$-linear form $T$ over any finite field,
    \[\PR(T) = \Ot_{d}(\AR(T)) .\]
%
\end{theorem}

The \emph{polarization} of a form $g \in \Forms_k$, which we denote by $\polar{g}$, 
is the $k$-linear form $\polar{g} = \D^k g$, given by the discrete derivative of order $k$; that is, $\polar{g}(\vec{x}) = \D^k_{\vec{x}} g$ where $\vec{x}=(x^{(1)},\ldots,x^{(k)})$.
(Note that if $g \colon \F^n \to \F$ then $\polar{g}\colon (\F^n)^k \to \F$; indeed, $\D^k_{\vec{x}} g$ depends only on $\vec{x}$, being independent of the point at which we take the discrete derivative, as $k=\deg(g)$.)
If $k < \ch(\F)$, $g$ can be recovered from its polarization: $g(x)=\polar{g}(x,\ldots,x)/k!$ (and $\overline{g}$ is the unique symmetric $k$-linear form satisfying this); see Fact~\ref{fact:diagonal} below.
%
For example, for $g(x)=x_1x_2$, 
we have $\polar{g}(x,y) = x_1y_2 + y_1x_2$,
and for $g(x)=x_1^2$
we have $\polar{g}(x,y) = 2x_1y_1$ (so $\polar{g}(x,x)$ fails to recover $g$ only when $\ch(\F)=2$).

We denote by $\partial_c f$ the formal derivative of a polynomial $f$ in direction $c$,
and denote by $\Nabla f$ the formal gradient of $f$, so that $\partial_c f(x) = c \cdot \Nabla f$.
Note that the formal derivative is well defined for polynomials over any field; by linearity, it suffices to define it for monomials: for $f=\prod_i x_i^{d_i}$, $\frac{\partial f}{\partial x_j} := d_j x_j^{d_j - 1} \prod_{i\neq j}x^{d_i}$, $\Nabla f := (\frac{\partial f}{\partial x_i} )_{i=1}^n$ and
$\partial_c f = c \cdot \Nabla f = \sum_{i=1}^n c_i \frac{\partial f}{\partial x_i}$.

\subsection{Proof outline}

We start by sketching the standard proof of $\GI_d(d)$.
The argument is roughly as follows. 
For $f$ of degree $d$, the $d$-th derivative $\Delta^d f$ is a degree-$d$ multilinear form, which is biased by assumption. Call that bias $\delta$. Using a partition-versus-analytic rank bound, such as Theorem~\ref{thm:SvR}, one obtains a partition-rank decomposition for $\Delta^df$ of length $\Ot_d(\log_{|\F|}(1/\d))$. 
By plugging in the diagonal $(x,\ldots,x)$ and dividing by $d!$, we obtain a decomposition of $f_d$ (the degree-$d$ component of $f$), meaning 
$\rk(f_d) \le \Ot_d(\log_{|\F|}(1/\d))$.
Writing $f_d = \sum_{i} \a_i\b_i$, a standard Fourier/pigeonhole-type argument gives a correlation for $f_d$ of at least $\delta^{\Ot_d(1)}$ with some linear combination of $\a_i$ and $\b_i$---which is of degree at most $d-1$---and thus a correlation for $f$ itself with such a lower-degree polynomial.

The proof of Theorem~\ref{thm: homogeneous GI 2d}, for homogeneous polynomials, follows this classical recipe, except it utilizes a new correlation lemma (Lemma~\ref{lemma:avg-cor}) to obtain a correlation with a linear combination of only the lower-degree factors from each product $\a_i\b_i$ in the decomposition, which is a polynomial of degree at most $\deg(f)/2$ ($<d$).

The difficulty in the non-homogeneous case of Theorem~\ref{thm:d+1-PGI-poly} 
is twofold. First, $\D^d f$ is no longer homogeneous of degree $d$, so a short decomposition of $\D^d f$ is not immediately available from its high bias. 
A significant part of this paper (Sections~\ref{sec:rk*} and~\ref{sec:derivatives}) is dedicated to the task of carefully obtaining a special decomposition for a general $f$ of degree $d+1$ which---while not being necessarily a homogeneous decomposition---can be used  to obtain a correlation with a polynomial of degree at most $\deg(f)-2$. We briefly sketch now how this is done.

In Section \ref{sec:rk*}, we introduce a new notion of rank for polynomials called $\rk^*$. The objective here was to find a notion of rank for which the following holds: whenever $\rk^{*}(f)$ is small, one can find a large correlation of $f$ with a polynomial of degree at most $\deg(f)-2$ (whereas for the usual notion of rank, one can only generally get $\deg(f)-1$). It turns out that the requirement needed is quite weak: for a decomposition $f=\sum_{i=1}^r \a_i\b_i$ ($1\le\deg(\a_i)\le\deg(\b_i)<\deg(f)$) to be a $\rk^*$-decomposition, we require the degree-$1$ polynomials $\a_i$ (in fact, only those with $\deg(\a_i\b_i)=\deg(f)$) to have linearly independent linear components. This requirement guarantees that the zero-locus of these polynomials is fairly large after an appropriate perturbation (Lemma~\ref{lemma:pert}), which is essential for us to obtain significant correlation.

We now explain how the $\rk^*$-decomposition of $f$ is obtained.
Write $g=f_{d+1}$ and $h=f_d$ for the homogeneous components of degree $d+1$ and $d$ respectively.
Proposition~\ref{prop:dd} provides a simple formula for $\D^d g$.
Proposition \ref{prop:derive-bias} in particular shows that decomposing $g$ by the standard polarization process goes smoothly, thanks to a chain rule for bias (Lemma~\ref{lemma:bias-chain}). The decomposition of $h$ is a more delicate issue, since the polarization $\polar{h}$ need not be biased. However, we prove that $\polar{h}$ is biased when adding an appropriate partial derivative $\partial_c\polar{g}$ (with respect to one of the $d$ vector variables); moreover, the sum $\polar{h} + \partial_c\polar{g}$ is a $d$-linear form, hence an application of Theorem~\ref{thm:SvR} gives a decomposition to this sum. To deal with the summand $\partial_c\polar{g}$ we added, we use a chain rule for $\rk^*$ (Proposition~\ref{prop:rk*-inh}) and the fact that polarization and formal differentiation commute (Lemma~\ref{lemma:Delta-Nabla}).

All of the above allows us to obtain a non-homogeneous decomposition for $f$, which is nevertheless useful for us since
it is a short $\rk^*$-decomposition.
Using the control we have over its linear factors, Proposition~\ref{prop:corr} allows us then to use this special decomposition to obtain significant correlation with a polynomial of degree $\deg(f)-2 \le d-1$, as needed.

We finally remark that we cannot utilize the arguments leading to the polynomial bound for $\GI_d(d)$, in Theorem~\ref{theorem: Milicevic multilinear variety},
in order to obtain a similar bound for degree $d+1$, 
since the latter does not help in providing a decomposition, which our arguments require.

\section{Correlation lemma, the homogeneous and the degree-$d$ cases}\label{sec: corr lemma, homogeneous GI}


In this section we prove a correlation lemma which will be important in the rest of the paper. As an almost immediate corollary, we will obtain Theorem~\ref{thm: homogeneous GI 2d}.

For a tuple of functions $\vec{A}=(A_1,\ldots,A_m)$, their joint zero set is $\Z(\vec{A}) = \{x \vvert \forall i \colon A_i(x) = 0\}$.
For an event $E$, we abbreviate $\bi(f \vvert E) = \Exp[\chi(f) \vvert E]$.

\begin{lemma}[average correlation lemma]\label{lemma:avg-cor}
    Let $f,A_1,\ldots,A_m \colon \F^n \to \F$ be any functions,
    and put $\vec{A}=(A_1,\ldots,A_m)$.
    Then
    \[\Exp_{c \in \F^m} \bi\Big(f - \sum_{i=1}^m c_iA_i\Big) 
    = \P(\Z(\vec{A})) \cdot \bi(f \vvert \Z(\vec{A})).\]
\end{lemma}
For completeness, if $\Z(\vec{A})=\emptyset$ then the right hand side should be interpreted as $0$.

\begin{proof}
    By direct computation:
    \begin{align*}
        \Exp_{c \in \F^m} \bi\Big(f - \sum_{i=1}^m c_iA_i\Big)        
        &= \Exp_{x} \Exp_{c \in \F^m} \chi\Big(f(x) - \sum_{i=1}^m c_iA_i(x)\Big)\\
        &= \Exp_{x} \chi(f(x)) \Exp_{c \in \F^m} \chi\Big(\sum_{i=1}^m c_iA_i(x)\Big)\\
        &= \Exp_{x} \chi(f(x)) \1_{\Z(\vec{A})}(x)\\
        &= |\F|^{-n}\sum_{x \in \Z(\vec{A})} \chi(f(x)),
    \end{align*}
    where the third equality is the fundamental fact that $\Exp_c \chi(c\cdot a)$ is $1$ if $a=0$, and otherwise $0$ (equivalently, the orthogonality of additive characters).
    If $\Z(\vec{A})\neq\emptyset$, we therefore get
    \[\Exp_{c \in \F^m} \bi\Big(f - \sum_{i=1}^m c_iA_i\Big)
        = \P(\Z(\vec{A}))\frac{1}{|\Z(\vec{A})|}\sum_{x \in \Z(\vec{A})} \chi(f(x))\\
        = \P(\Z(\vec{A}))\bi(f \vvert \Z(\vec{A})). \qedhere\]
\end{proof}

We will also need
the following Theorem due to Warning (for a short proof see~\cite{Asg18}).
\begin{theorem}[Warning's second theorem~\cite{War35}]\label{theorem: warning}
    For all polynomials $f_1,\ldots,f_m \colon \F^n \to \F$ 
    over a finite field $\F$,
    if $\Z(f_1,\ldots,f_m) \neq \emptyset$ then
    $|\Z(f_1,\ldots,f_m)| \ge |\F|^{n-\sum_{i=1}^m \deg(f_i)}$.
\end{theorem}
We obtain the following correlation bound, which only makes an assumption on the zero set of $f$ (or on an ideal containing $f$).
Recall the notation $\cor_{<d}(f) = \max_{\deg(P)< d} \cor(f,P)$.

\begin{corollary}\label{coro:corr}
    If $\Z(\vec{A}) \sub \Z(f)$
    for some $\vec{A}=(A_1,\ldots,A_m) \in \Forms(\F)^m$
    with $1 \le \deg(A_i) < d$,
    then $\cor_{<d}(f) \ge |\F|^{-(d-1)m}$.
\end{corollary}
\begin{proof}
    Put $D = \sum_{i=1}^m \deg(A_i)$. 
    Apply Lemma~\ref{lemma:avg-cor}
    to obtain
    \begin{align*}
    \Exp_{c \in \F^m} \bias\Big(f - \sum_{i=1}^m c_iA_i\Big) 
    &\ge \Big| \Exp_{c \in \F^m} \bi\Big(f - \sum_{i=1}^m c_iA_i\Big) \Big|\\
    &= \P(\Z(\vec{A})) \cdot |\bi(f \vvert \Z(\vec{A}))|
    = \P(\Z(\vec{A})) \ge |\F|^{-D},
    \end{align*}
    where the last inequality applies Warning's second theorem (Theorem~\ref{theorem: warning}), using the fact that $0 \in \Z(A_1,\ldots,A_m)$ by homogeneity.
    Thus, some $P=\sum_{i=1}^m c_iA_i \in \Poly_{d-1}$ 
    satisfies $\bias(f - P) \ge |\F|^{-D} \ge |\F|^{-(d-1)m}$.    
\end{proof}

We now prove Theorem \ref{thm: homogeneous GI 2d}.
\begin{proof}[Proof of Theorem \ref{thm: homogeneous GI 2d}] 
Let $f: \F^n \to \F$ be a homogeneous polynomial with $k:=\deg(f) < 2d$,
where $\F=\F_q$ with $\ch(\F_q)>k$,
and put $\bias(\D^d f) = \d$.
We assume $k \ge d \ge 2$ and $\d>0$, as otherwise there is nothing to prove.
First, we observe that $\bias(\polar{f}) \ge \d^{2^d}$,
recalling that $\polar{f}=\D^k f$ denotes polarization.
Indeed, using the monotonicity of Gowers norms, 
$\norm{g}_{U^k} \ge \norm{g}_{U^d}$ for $k\ge d$ and any $g \colon \F^n \to \CC$,
and using the fact that, by definition, 
$\bias(\D^t f) = \norm{e(f)}_{U^t}^{2^t}$ for any $t$,
we have that
\[\bias(\polar{f}) = \bias(\D^k f) = \norm{e(f)}^{2^k}_{U^k} 
\ge \norm{e(f)}^{2^k}_{U^d} = \bias(\D^d f)^{2^{k-d}} \ge \d^{2^d}.\]
By Theorem~\ref{thm:SvR} applied on the $k$-linear form $\polar{f}$,
we obtain a partition-rank decomposition
\[\polar{f} = \sum_{i=1}^r Q_i R_i\]
with $Q_i,R_i$ multilinear forms of degrees at most $k-1$, on disjoint sets of vector variables, 
and 
\[r \le \Ot_k(\log_q(1/\bias(\polar{f}))) = \Ot_d(\log_{q}(1/\d)).\]

Assume without loss of generality that $\deg(Q_i) \le \deg(R_i)$ for every $i$.
Since $\deg(Q_i)+\deg(R_i) = k$, we have $\deg(Q_i) \le k/2$.
Thus, we can put $\polar{f}$ in the ideal
$\polar{f} \in \langle Q_1,\ldots,Q_r \rangle$.
Plug in the diagonal: let $A_i(x) = Q_i(x,\ldots,x)$, 
and observe that $f(x) = \frac{1}{k!}\polar{f}(x,\ldots,x)$, 
using the homogeneity of $f$ and $\ch(\F)>k$.
It follows that $f \in \langle A_1,\ldots,A_r \rangle$ with each $A_i$ homogeneous of degree at most $k/2 < d$, by assumption.
We may assume $\deg(A_i) \ge 1$, as otherwise $A_i=0$ so we may discard it.
Since $\Z(\vec{A}) \sub \Z(f)$,
we are now done by Corollary~\ref{coro:corr}:
\[\cor_{<d}(f) \ge q^{-d r} \ge q^{-\Ot_d(\log_{q}(1/\d))} = \d^{\Ot_d(1)}. \qedhere\]
\end{proof}

\subsection{A remark on the case of degree-$d$ polynomials}\label{subsection: degree d}

In~\cite{Mil25}, Mili\'{c}evi\'{c}~ proved a ``weak structural result for biased multilinear forms'' (see Theorem~\ref{theorem: Milicevic multilinear variety} below).
In this section we show that it easily implies the degree-$d$ case of the Polynomial Gowers Inverse conjecture, using our correlation lemma, Lemma~\ref{lemma:avg-cor}.
For comparison, the degree-$d$ case was proved in~\cite{MR25} using an approximation result (Theorem~2 in that paper) which roughly states that the phase of any multilinear form of large bias is close in the $L^2$-norm to a linear combination of phases of lower-degree forms.

We first state the aforementioned weak structural result.
Note that if a multilinear form $T \colon (\F^n)^d \to \F$ is of degree $k$ then $T=T(x^{(i_1)},\ldots,x^{(i_k)})$ depends on some $k$ of the $d$ variable vectors $x^{(1)},\ldots,x^{(d)}$ 
(that is, $T \colon (\F^n)^I \to \F$ for some subset $I \sub [d]$ with $|I|=k$).

Call a subset $V \sub (\F^n)^d$ a \emph{multilinear variety} if $V=\Z(T_1,\ldots,T_m)$ is the zero set of a collection of multilinear forms $T_i$ of degree at most $d$.
With a slight abuse of notation, the \emph{codimension} of $V$ is the least number $m$ of multilinear forms needed.
The \emph{density} of $V$ is $|V|/|\F^{nd}| \in [0,1]$.

\begin{theorem}[\cite{Mil25}]\label{theorem: Milicevic multilinear variety}
    Any multilinear variety $V \subseteq (\F_p^n)^d$ 
    of density $\d$ 
    contains a multilinear variety of codimension at most $O_d(\log_{p}(1/\d) + 1)$.

\end{theorem}

We now give our short alternative proof to the Polynomial Gowers inverse conjecture in degree $d$, namely, the bound $\cor_{<d}(f) \ge \bias(\D^d f)^{O_d(1)}$ for every polynomial $f \in \Poly_d(\F_p)$ with $d<p$.

\begin{proof}[Proof of $\GI_d(d)$ from Theorem~\ref{theorem: Milicevic multilinear variety}]
    Let $f \colon \F^n \to \F$ with $\deg(f) \le d$ and $\bias(\D^d f) = \d$,
    where $\F=\F_p$ and $d<p$.
    We assume $\deg(f)=d \ge 2$ and $\d>0$, as otherwise there is nothing to prove.
    Recall that $\D^d f \colon (\F^n)^d \to \F$ is a $d$-linear form, 
    and let us write $\D^d f(x^{(1)},\ldots,x^{(d)}) = \sum_{i=1}^n x^{(d)}_i T_i(x^{(1)},\ldots,x^{(d-1)})$, 
    so that each $T_i\colon (\F^n)^{d-1} \to \F$ is a $(d-1)$-linear form.
    Put $\vec{T} = (T_1,\ldots,T_n)$,
    and consider the multilinear variety $V = \Z(\vec{T}) \sub (\F^n)^{d-1}$. 
    Note that the density of $V$ is $\P_{\vec{x} \in (\F^n)^{d-1}}(\vec{x} \in V) = \bias(\D^d f) = \d$;
    indeed, this is the standard observation 
    $\bias(\D^d f) 
    = \Exp_{\vec{x}=(x^{(1)},\ldots,x^{(d-1)})} \Exp_{x^{(d)}} \chi(x^{(d)} \cdot \vec{T}(\vec{x})) 
    = \Exp_{\vec{x}} \1_V(\vec{x})
    = \P_{\vec{x}}(\vec{x} \in V)$.
    By Theorem~\ref{theorem: Milicevic multilinear variety} with $d-1$, there are multilinear forms $\vec{S}=(S_1,\ldots,S_m)$, with each $S_i \colon (\F^n)^{d-1}\to\F$ of degrees at most $d-1$, 
    such that $\Z(\vec{S}) \sub \Z(\vec{T})$ 
    and 
    $m \le O_d(\log_{p}(1/\d)+1)$.
    We observe that for any nonzero $d$-linear form $T$, $\AR(T) \ge 2^{-d}$; this follows for example from Claim~3.3 in~\cite{Lov19}.
    Since $\D^d f \neq 0$, as otherwise $\deg(f) < d$,
    we have that $\log_{p}(1/\d) = \AR(\D^d f) \ge 2^{-d}$.
    Thus, $\log_{p}(1/\d)+1 \le (2^d+1)\log_{p}(1/\d)$,
    and therefore 
    \begin{equation}\label{eq:poly-many-eqs}
        m \le O_d(\log_{p}(1/\d)).
    \end{equation}

    Plugging in the diagonal, 
    let $A_i(x) = S_i(x,\ldots,x)$, 
    homogeneous with $\deg(A_i) \le \deg(S_i) \le d-1$.
    Put $\vec{A}=(A_1,\ldots,A_m)$.
    We have $f_d(x) = (1/d!)\D^d f(x,\ldots,x)$
    using $d<p$.
    Since $\Z(\vec{S}) \times \F^n \sub \Z(\vec{T}) \times \F^n \sub \Z(\D^d f)$,
    we have that $\Z(\vec{A}) \sub \Z(f_d)$.\footnote{Indeed, $\Z(\vec{A}) = \{x \setmid \vec{A}(x)=0\}
    = \{x \setmid \vec{S}(x,\ldots,x)=0\}
    \sub \{x \setmid \D^d f(x,\ldots,x,x)=0\}
    = \Z(f_d)$.}

    We may assume $\deg(A_i) \ge 1$, as otherwise $A_i=0$ so we may discard it.
    By Corollary~\ref{coro:corr}, 
    $\cor(f_d, P) \ge p^{-dm}$
    for some polynomial $P=\sum_{i=1}^m c_iA_i$ of degree at most $d-1$.
    Note that $\cor(f,\, P+f_{<d}) = \bias(f-(P+f_{<d})) = \bias(f_d - P)$. 
    Thus,
    \[\cor_{<d}(f) \ge 
    \bias(f_d - P) \ge p^{-dm}
    \ge p^{-O_d(\log_p(1/\d))} = \d^{O_d(1)}\]
    using~(\ref{eq:poly-many-eqs}) in the last inequality. This completes the proof.

\end{proof}

\section{$\rk^*$}\label{sec:rk*}

In this section we introduce our refined notion of rank, compare it to standard notions from the literature, and prove several properties that would be useful later.

Recall that for a form $f$, 
the rank of $f$ is the smallest number of reducible forms that sum to $f$.
A natural extension of rank to general polynomials (that is, not necessarily homogeneous) expresses $f$ as a sum of reducible polynomials of degree at most $\deg(f)$ (so the factors are of degree below $\deg(f)$).
Our new notion of rank does that, except it also adds a mild condition on the factors in the decomposition, or more specifically, on its factors of degree $1$.
Intuitively, one can think of this as a very weak ``quasirandomness'' condition on the decomposition.

Call a set of degree-$1$ polynomials \emph{affinely independent}
if their linear components 
are linearly independent.
For example, $x$ and $x+1$ are linearly independent but not affinely independent, since they have the same linear component, namely, $x$.

\begin{definition}[$\rk^*$]\label{def:rk*}
    For a non-constant polynomial $f$, $\rk^*(f)$ is the smallest $r$ such that:
    \begin{itemize}
        \item $f = \sum_{i=1}^r f_i$ for some reducible polynomials $f_i=\a_i\b_i$ with $\deg(f_i) \le \deg(f)$,\footnote{That is, $\deg(\a_i),\deg(\b_i) \ge 1$ and $\deg(\a_i)+\deg(\b_i) \le \deg(f)$. Note that $\deg(\a_i), \deg(\b_i) < \deg(f)$.} and
        \item $\{\a_i \vvert \deg(\a_i)=1, \deg(f_i) = \deg(f)\}$ are affinely independent, where $\deg(\a_i) \le \deg(\b_i)$.
    \end{itemize}
\end{definition}
If $f$ is a constant polynomial, we set $\rk^*(f)=0$.
By convention, $\rk^*(f) = \infty$ if no such decomposition exists.
This happens when $\deg(f)=1$, 
and can also happen when $\deg(f)=2$; for example, for any irreducible univariate quadratic.\footnote{Affine independence implies at most $n$ top-degree summands having a degree-$1$ factor. 
For a univariate quadratic, this means at most $n=1$ summands.}
On the other hand, $\deg(f)\ge 3$ implies $\rk^*(f) < \infty$; in fact, $\rk^*(f) \le n+1$.\footnote{Indeed, let $q$ be any reducible quadratic with affine part $q_{\le 1}=f_{\le 1}$ (say $q(x)=(x_1+1)(f_{\le 1}(x)-f_0x_1)$). 
Then $f-q$ has only monomials of degree at least $2$, so it has a $\rk^*$-decomposition with $\a_1=x_1,\ldots,\a_n=x_n$.
Since $\deg(q)<3\le\deg(f)$, this gives a $\rk^*$-decomposition for $f$ of length at most $n+1$.
(This bound cannot be improved to $n$ in general; e.g., any  irreducible univariate $f$, assuming $\F$ is not algebraically closed, has $\rk^*(f) \ge 2=n+1$.)}


Note that if $f$ is homogeneous then $\rk^*(f)=\rk(f)$.
Indeed, $\rk^*(f)\le\rk(f)$ since the factors may be assumed to be linearly independent forms (see Lemma~\ref{lemma:rk-lin-ind}), so the affine independence condition trivially holds;
conversely, $\rk(f)\le\rk^*(f)$ since for each summand $\a_i\b_i$, the lower-degree components of $\a_i$ and of $\b_i$ contribute only monomials of degree below $\deg(f)$ to the product, and thus may all be discarded.
Note too that $f$ of positive degree is irreducible if and only if $\rk^*(f)>1$, similarly to $\rk$.

The importance of $\rk^*(f)$ stems from the fact that it is closely related to correlation with polynomials of degree $\deg(f)-2$, whereas the usual $\rk(f)$ only gives information about correlation with degree $\deg(f)-1$.

As an example for $\rk^*$, consider a polynomial of the form $f = (A-B)\cdot\ell + B$, where $\deg(A)=\deg(B)>1$ and $\deg(\ell)=1$.
While $f$ has the decomposition $f = \ell A + (-\ell+1)B$, 
this is not a $\rk^*$-decomposition since its degree-$1$ factors $\ell$ and $-\ell+1$ are not affinely independent. 
In fact, $\rk^*(f)$ may be arbitrarily large.
Thus, it is not surprising that $f$ may have negligible correlation with polynomials of degree at most $\deg(f)-2$ (see Remark~\ref{remark:rk*-vs-rank} next).
This is in sharp contrast to polynomials with low $\rk^*$, which do have significant correlation with polynomials of degree at most $\deg(f)-2$ (see Proposition~\ref{prop:corr} below).

\begin{remark}\label{remark:rk*-vs-rank}
    Interestingly, $\rk^*(f)$ can be more closely related to $\bias(f)$, or rather to analytic rank $\AR(f)=-\log_{|\F|}(\bias(f))$, than other standard notions of rank in the literature.
    For concreteness, consider the notion used by Green and Tao in \cite{GT09}: $\rank(f)$ is the least number $k$ of polynomials $\a_1,\ldots,\a_k$ with $\deg(\a_i) < \deg(f)$ such that $f \equiv \Gamma(\a_1,\ldots,\a_k)$ for some function $\Gamma \colon \F^k \to \F$.
    Consider the polynomial $f(x,y) = (A-B)\ell + B$ with $\ell=\ell(y)$ of degree $1$, and $A=A(x)$, $B=B(x)$ random of constant degree $d\ge 2$ on $n$ variables.
        Clearly, $\rank(f) \le 3$.
        In contrast, we will next show (slightly informally) that $\rk^*(f)$ and $\AR(f)$ grow with $n$, the number of $x$ variables.
        We note that a random $n$-variate polynomial over $\F_q$ of degree $d$ has exponential small $q^{-\Omega_d(n)}$ correlation with lower-degree polynomials~\cite{BHL12,BGY20}.
        
        Consider a $\rk^*$-decomposition $f=\sum_{i=1}^{\rk^*(f)} \a_i\b_i$.
        By affine independence, setting to $0$ all degree-$1$ factors $\a_i$ coming from top-degree summands gives an affine subspace that is not empty.
        Let $U$ be the non-empty intersection of this affine subspace with the affine hyperplane $\ell=c$, for an appropriate $c \in \F$. 
        Note that $\codim(U) \le \rk^*(f)+1$. Observe that the restriction $f'$ of $f$ to $U$ has (in the sense of Green-Tao) $\rank(f') \le 2\rk^*(f)$, since $\deg(f')=d$ (with high probability) and it inherits a decomposition in terms of polynomials of degree at most $\deg(f)-2 = d-1$.
        By construction, $f'$ is a nonzero linear combination of $A$ and $B$ restricted to $U$; using our bound on $\codim(U)$ and the randomness of $A$ and $B$,
        $\rank(f')$ grows with $n$, and thus also $\rk^*(f)$.

        We will next see that $\cor(f,P) \le |\F|^{-\Omega(n)}$ for any $P \in \Poly_{\deg(f)-2}$. 
        Note that
        \[\cor(f,P) = |\Exp_y \Exp_x \chi(f(x,y)-P(x,y))|
        \le \max_y |\Exp_x \chi(f(x,y)-P(x,y))|
        = \max_y \cor(f_y,P_y)\]
        where $f_y(x)=f(x,y)$ and $P_y(x)=P(x,y)$.
        For \emph{every} fixed $y$, we have by construction that $f_y$ is a nontrivial linear combination of two random polynomials of degree $d$, 
        and thus is a random polynomial of degree $d$ as well; therefore, since $\deg(P_y) \le \deg(P) \le d-1$, we have that
        $\cor(f_y,P_y) \le |\F|^{-\Omega(n)}$. By the above, the same follows for $\cor(f,P)$.\footnote{An alternative proof that $\rk^*(f)$ grows with $n$ follows immediately by combining this correlation upper bound together with Proposition~\ref{prop:corr} below.}
        
        Finally, note that this correlation upper bound in particular applies to $\bias(f)=\cor(f,0)$,
        so $\AR(f) \ge \Omega(n)$, similarly to $\rk^*(f)$.

\end{remark}

\subsection{Properties of $\rk^*$}

In general, $\rk^*$ is \emph{not} subadditive, unlike other notions of rank. 
That is, the sum of two $\rk^*$-decompositions need not be another $\rk^*$-decomposition, for the simple reason that there could be linear dependencies between the linear forms in the degree-$1$ factors of the two decompositions.
For example, the $\rk^*$ of $(x+1)R$ and of $xS$, where $\deg(R),\deg(S) \ge 2$, is $1$, but the $\rk^*$ of their sum $(x+1)R + xS$ need not be at most $2$; in fact, it can be unbounded (recall Remark~\ref{remark:rk*-vs-rank}).

That being said, $\rk^*$ is subadditive if the two polynomials are of different degrees.

\begin{observation}[subadditivity]\label{obs:subadd}
    $\rk^*(g+h) \le \rk^*(g) + \rk^*(h)$ for positive $\deg(h) \neq \deg(g)$.
\end{observation}

A basic fact about rank is that it is \emph{not} inherited by partial derivatives: 
even if a form has low rank, a partial derivative might not.
Consider, say, reducible forms; for example, we have $\rk(x \cdot S) \le 1$ for any form $S=S(y)$, yet the partial derivative with respect to $x$ is $S(y)$, which may have arbitrarily high rank.

That being said, we show that, perhaps surprisingly,
\emph{adding} to a polynomial a partial derivative does not increase its $\rk^*$, up to a polynomial of lower degree.\footnote{Alternatively, one could omit the lower degree polynomial $\g$ from Lemma~\ref{claim:derive-rk*}, and from Proposition~\ref{prop:rk*-inh} below, by redefining $\rk^*(f)$ to ignore the components of $f$ of degree at most $\deg(f)-2$.}

\begin{lemma}[derivation invariance]\label{claim:derive-rk*}
    For any 
    $g \in \Forms_k$ with $k \ge 2$, and any direction vector $c$,
    \[\rk^*(g+\partial_c g + \g) \le \rk(g)\]
    for some $\g \in \Poly_{k-2}$.
\end{lemma}
For example, for the univariate $g(x)=x^2 \in \Forms_2$,
we have $\partial_c g=2cx$,
and it is easy to see that there is a constant $\g \in \F$ such that $g+\partial_c g+\g$ is reducible (is of $\rk^*$ $1$); 
indeed, take $\g=c^2$.
Lemma~\ref{claim:derive-rk*} can be thought of as a generalization of such identities.


Combining the above two properties, we show that $\rk^*$ satisfies a certain approximate ``chain rule'', 
which would complement a chain rule for bias that we prove later in Section \ref{sec:main-pf}.

\begin{proposition}[chain rule for $\rk^*$]\label{prop:rk*-inh}
    For any $g \in \Forms_k$ with $k \ge 2$ and $h \in \Poly_{<k}$,
    and any direction vector $c$,
    \[\rk^*(g+h+\g) \le \rk(g) + \rk^*(h-\partial_c g) \]
    for some $\g \in \Poly_{k-2}$. 
\end{proposition}

We now prove the above properties of $\rk^*$.
We start with subadditivity.

\begin{proof}[Proof of Observation~\ref{obs:subadd}]
    Put $r_1=\rk^*(g)$, $r_2 = \rk^*(h)$,
    and assume both are finite as otherwise there is nothing to prove.
    Let $g=\sum_{i=1}^{r_1} g_i$ and $h=\sum_{i=1}^{r_2} h_i$, with $g_i$ and $h_i$ reducible polynomials with $\deg(g_i) \le \deg(g)$ and $\deg(h_i) \le \deg(h)$, be decompositions witnessing $\rk^*(g)$ and $\rk^*(h)$, respectively (recall that $g$ and $h$ are not constant).
    We have $g+h = \sum_{i=1}^{r_1} g_i + \sum_{i=1}^{r_2} h_i$.
    Assume without loss of generality that $\deg(h) < \deg(g)$.
    Since $\deg(h_i) \le \deg(h)$, 
    observe that the affine independence condition for $\rk^*(g+h)$ applies only for the $g_i$ (of degree $\deg(g)$ with a degree-$1$ factor), and is therefore satisfied by the definition of $\rk^*(g)$.
    Thus, $\rk^*(g+h) \le r_1 + r_2 = \rk^*(g) + \rk^*(h)$, as needed.
\end{proof}

We will use the following easy observation.
\begin{lemma}\label{lemma:rk-lin-ind}
    For every $g \in \Forms$ and a shortest $\rk$-decomposition $g = \sum_i \a_i\b_i$, 
    the forms $(\a_i)_i$ (and $(\b_i)_i$) are linearly independent. 
\end{lemma}
\begin{proof}
    Put $r=\rk(g)$, 
    and let $g = \sum_{i=1}^r \a_i\b_i$ be a $\rk$-decomposition.
    Suppose for contradiction that $\a_r = \sum_{i=1}^{r-1} c_i\a_i$ with $c_i$ scalars. By homogeneity, we may assume that $\deg(\a_i) = \deg(\a_r)$ whenever $c_i \neq 0$.
    Observe that $g$ has the decomposition $g = \sum_{i=1}^{r-1} \a_i(\b_i + c_i\b_r)$.
    Each summand is a reducible form, since $\a_i$ is homogeneous as well as $\b_i + c_i\b_r$; indeed, $\deg(\b_i)=\deg(g)-\deg(\a_i) = \deg(g)-\deg(\a_r) = \deg(\b_r)$ whenever $c_i \neq 0$. Thus, $r = \rk(g) \le r-1$, a contradiction.
\end{proof}

Next we prove the invariance of $\rk^*$ with respect to adding a derivative.

\begin{proof}[Proof of Lemma~\ref{claim:derive-rk*}]
    Put $r=\rk(g)$ 
    and write $g = \sum_{i=1}^r \a_i\b_i$ with forms $\a_i,\b_i$ 
    satisfying $1\le\deg(\a_i) \le \deg(\b_i) < k$ and $\deg(\a_i\b_i)=k$.
    
    By the product rule, 
    $\partial g_c = \sum_{i=1}^r \a_i(\partial_c\b_i) + (\partial_c \a_i)\b_i$.
    Observe that we 
    have the following fortunate identity:
    \[g+\partial_c g 
    = \sum_{i=1}^r (\a_i+\partial_c \a_i)(\b_i+\partial_c \b_i) 
    - \sum_{i=1}^r (\partial_c \a_i)(\partial_c \b_i).\]
    Put $\a_i' = \a_i+\partial_c \a_i$ and $\b_i' = \b_i+\partial_c \b_i$, 
    and note that $\deg(\a_i')=\deg(\a_i) \le \deg(\b_i)=\deg(\b_i')$.
    Thus, the linear components of those factors $\a'_i$ with $\deg(\a'_i)=1$
    are precisely the $\a_i$ with $\deg(\a_i)=1$,
    which by Lemma~\ref{lemma:rk-lin-ind} are linearly independent.
    As for the remaining terms, note that $\deg((\partial_c \a_i)(\partial_c \b_i)) \le \deg(\a_i\b_i)-2 = k-2$,
    so their sum $\g$ has $\deg(\g) \le k-2$.
    Thus, we have a (non-homogeneous) $\rk^*$-decomposition of $g+\partial_c g +\g$, and thus $\rk^*(g+\partial_c g+\g) \le r = \rk(g)$, as needed.
\end{proof}

We are now ready to prove our chain rule for $\rk^*$.
\begin{proof}[Proof of Proposition~\ref{prop:rk*-inh}]
    Apply Lemma~\ref{claim:derive-rk*} on $g \in \Forms_k$ and $c$ to obtain $\g \in \Poly_{k-2}$ with
    $\rk^*(g+\partial_c g + \g) \le \rk(g)$.
    If $h-\partial_c g$ is a constant $a$, then we are done since $\g-a \in \Poly_{k-2}$ and
    \[\rk^*(g+h+(\g-a)) = \rk^*(g+\partial_c g+\g) \le \rk(g) = \rk(g)+\rk^*(h-\partial_c g).\]
    Otherwise, 
    \begin{align*}
        \rk^*(g+h+\g) &= \rk^*\big( (g+\partial_c g+\g) + (h-\partial_c g) \big)\\
        &\le \rk^*(g+\partial_c g+\g) + \rk^*(h-\partial_c g) 
        \le \rk(g) + \rk^*(h-\partial_c g)
    \end{align*}
    where the first inequality uses the subadditivity of $\rk^*$ in Observation~\ref{obs:subadd}, as the two polynomials have different positive degrees: $k$ and at most $k-1$.
%
\end{proof}

\subsection{Lower-degree correlation}\label{sec:corr}

Recall that standard correlation results show that if a polynomial $f$ has a short decomposition in terms of lower degree polynomials, then it has a significant correlation with a polynomial of degree at most $\deg(f)-1$.
In this section we show that if a polynomial $f$ has a short $\rk^*$-decomposition, 
then it has a significant correlation with a polynomial of degree at most $\deg(f)-2$.\footnote{Assuming $\deg(f)\ge 3$. If $\deg(f)=2$ we only get correlation with degree $\deg(f)-1$, naturally.}
This relies on our correlation lemma in Lemma~\ref{lemma:avg-cor}, but here we only assume small $\rk^*$.

\begin{proposition}[lower-degree correlation]\label{prop:corr}
    For any $f \in \Poly_k(\F_q)$ with $k \ge 3$,
    \[\cor_{<k-1}(f) \ge q^{-2\rk^*(f)} .\]
\end{proposition}
For completeness, if $\rk^*(f)=\infty$ then the right hand side should be interpreted as $0$.

In order to convert a short $\rk^*$-decomposition to significant correlation, via Lemma~\ref{lemma:avg-cor}, we use the following lemma which 
puts a shift of $f$ in an ideal generated by few polynomials of degree at most $\deg(f)-2$ that, moreover, frequently vanish together.

\begin{lemma}[perturbation lemma]\label{lemma:pert}
    For any $f \in \Poly_k(\F_q)$ with $k \ge 3$, 
    we have $f-c \in \langle \vec{A} \rangle$, where $\vec{A}=(A_1,\ldots,A_m)$ and $c \in \F_q$, such that:
    \begin{itemize}
        \item $1 \le \deg(A_i) \le k-2$,
        \item $\P(\Z(\vec{A})) \ge q^{-m}$, 
        \item and $m \le 2\rk^*(f)$.
    \end{itemize}
    
\end{lemma}

\begin{proof}
    Put $r=\rk^*(f)$ and $\F=\F_q$.
    Assume $r<\infty$, as otherwise we are done by taking $\vec{A}=(x_1,\ldots,x_n)$ and $c=f(0)$.
    Assume $f$ is non-constant, as otherwise we are done by taking $m=0$ and $c=f \in \F_q$.
    Write $f = \sum_{i=1}^r \a_i\b_i$ with 
    $1 \le \deg(\a_i) \le \deg(\b_i) \le k-1$ and $\deg(\a_i\b_i) \le k$.
    Assume without loss of generality that the polynomials $\a_i$ satisfying both $\deg(\a_i)=1$ and $\deg(\a_i\b_i)=k$ are precisely $\a_{t+1},\ldots,\a_r$ for some $t\le r$.
    Observe that for every $i \le t$ we necessarily have $\deg(\b_i) = \deg(\a_i\b_i)-\deg(\a_i) \le k-2$;
    indeed, we have that $\deg(\a_i\b_i) \le k-1$ or $\deg(\a_i) \ge 2$.
    
    Let $\vec{A}^1=(\a_1,\ldots,\a_t)$, $\vec{A}^2=(\a_{t+1},\ldots,\a_r)$, and $\vec{B}^1=(\b_1,\ldots,\b_t)$, 
    and $m=r+t$.
    We claim that there are $y^1,z^1 \in \F^t$ such that 
    the map $(\vec{A}^1,\vec{A}^2,\vec{B}^1) \colon \F^n \to \F^m$ satisfies
    \begin{equation}\label{eq:PP}
        \P\big((\vec{A}^1,\vec{A}^2,\vec{B}^1)=(y^1,0,z^1)\big)
        \ge q^{-m}.
    \end{equation}
    If $t=r$ (for example, when $\deg(f)<k$), that is, $\vec{A}^2$ is empty, then~(\ref{eq:PP}) follows by the pigeonhole principle, so henceforth assume otherwise.    
    By the definition of $\rk^*(f)$, the linear components of $\a_{t+1},\ldots,\a_r$ are linearly independent, which implies that the affine map $\vec{A}^2 \colon \F^n \to \F^{r-t}$ is perfectly equidistributed; 
    indeed, the associated linear map is surjective, so its fibers are all cosets of the kernel and thus of the same size.
    In particular, $\P(\vec{A}^2=0) = q^{-(r-t)}$ ($>0$).
    Since 
    \[\sum_{y,z \in \F^t} \P\big( (\vec{A}^1,\vec{B}^1)=(y,z) \vvert \vec{A}^2=0 \big) = 1,\]
    there exist $y^1,z^1 \in \F^t$ such that
    \[
        \P\big( (\vec{A}^1,\vec{B}^1)=(y^1,z^1) \vvert \vec{A}^2=0 \big) \ge q^{-2t}.
    \]
    This proves~(\ref{eq:PP}):
    \begin{align*}
        \P\big((\vec{A}^1,\vec{A}^2,\vec{B}^1)=(y^1,0,z^1)\big)
        &= \P(\vec{A}^2=0) \cdot \P\big((\vec{A}^1,\vec{B}^1)=(y^1,z^1) \vvert \vec{A}^2=0\big)\\
        &\ge q^{-(r-t)} q^{-2t} = q^{-m} .
    \end{align*}
    Now, put $y=(y^1,y^2) \in \F^r$ with $y^2=0 \in \F^{r-t}$.
    Decompose $f$ as follows (recall $y_i=0$ for $i > t$):
    \[f = \sum_{i=1}^r (\a_i-y_i)\b_i + \sum_{i=1}^t y_i(\b_i-z_i) + \sum_{i=1}^t y_iz_i .\]
    Let $A_i = \a_i-y_i$ for $i \le r$, let $A_{r+i}=\b_i-z_i$ for $i \le t$,
    and let $c=\sum_{i=1}^t y_iz_i$ ($\in \F$).  
    Then $f-c \in \langle A_1,...,A_m \rangle$ with $m=r+t \le 2r$. 
    Moreover, for $i \le t$ we have $\deg(A_{r+i}) = \deg(\b_i)$, 
    and $1 \le \deg(\b_i) \le k-2$, while for $i \le r$ we have $\deg(A_i) = \deg(\a_i)$, and $1 \le \deg(\a_i) \le \lfloor k/2 \rfloor \le k-2$ using our assumption $k \ge 3$. 
    %
    Observe that
    \[\P(\forall i \colon A_i = 0) = 
    \P\big((\vec{A}^1,\vec{A}^2,\vec{B}^1)=(y^1,y^2,z^1)\big) 
    \ge q^{-m},\]
    where the inequality is~(\ref{eq:PP}).
    This completes the proof.
\end{proof}

\begin{proof}[Proof of Proposition~\ref{prop:corr}]
    Apply Lemma~\ref{lemma:pert} on $f$ to obtain $m$ polynomials $\vec{A}=(A_1,...,A_m) \in \Poly_{k-2}^m(\F_q)$ and a scalar $c_0 \in \F_q$ such that 
    $f-c_0 \in \langle \vec{A} \rangle$,
    where $\P(\Z(\vec{A})) \ge q^{-m}$ and $m \le 2\rk^*(f)$.
    Now apply Lemma~\ref{lemma:avg-cor} on $f,A_1,\ldots,A_m$
    to deduce that 
    \begin{align*} 
        \Exp_{c \in \F^m} \Big|\bi\Big(f - \sum_{i=1}^m c_iA_i\Big) \Big|
        &\ge \Big|\Exp_{c \in \F^m} \bi\Big(f - \sum_{i=1}^m c_iA_i\Big) \Big|\\
        &= \P(\Z(\vec{A})) \cdot |\bi(f \vvert \Z(\vec{A}))|\\
        &\ge q^{-m} \cdot |\bi(c_0 \vvert \Z(\vec{A}))| \ge q^{-2\rk^*(f)}.
    \end{align*}
    Since $\sum_{i=1}^m c_iA_i \in \Poly_{k-2}$, 
    we deduce that $f$ is significantly correlated with a polynomial of degree at most $k-2$:
    \[\cor_{<k-1}(f) 
    \ge \max_{c \in \F^m} \cor\Big(f,\, \sum_{i=1}^m c_iA_i\Big)
    \ge q^{-2\rk^*(f)}. \qedhere\]
\end{proof}

\section{Discrete derivatives}\label{sec:derivatives}

For a multilinear form $T$, we denote
$\partial_c T(\v) = T(\v,c)$.
In other words, $\partial_c T$ is the directional derivative with respect to the (say) last  vector variable.
We denote $\Nabla T(\v) = \big( T(\v,e_i) \big)_i$,
so that $\Nabla T$ is the gradient with respect to the last vector variable and $\partial_c T = c \cdot \Nabla T$.

Our main result in this section is a particularly useful description of the order-$d$ discrete derivative of forms of degree $d+1$.

\begin{proposition}\label{prop:dd}
    For any $g \in \Forms_{d+1}(\F)$, where $\ch(\F) \neq 2$, and any $\vec{v} \in (\F^n)^d$,
    \[\D^d_{\vec{v}} g(x) = \polar{g}(\vec{v}, x')
    \quad\text{ where }
    x' = x+\frac12\sum_{t=1}^d v^{(t)}.\]
\end{proposition}
Put differently, $\D^d_{\vec{v}} g(x) = x'\cdot \Nabla \polar{g}(\vec{v})$.
We will also show that, although not entirely obvious, polarization and formal derivation commute. 

\begin{lemma}\label{lemma:Delta-Nabla}
    For any form $g$ and any direction vector $c$,
    \[\partial_c \polar{g} = \polar{\partial_c g}.\]
\end{lemma}

It is instructive to consider simple examples of these somewhat subtle identities.

\begin{example}[Example for Lemma~\ref{lemma:Delta-Nabla}]
    For $g(x) = x_1 x_2^2$,
    we have $\partial_c g(x) = c \cdot (x_2^2,\, 2x_1x_2)$,
    so that $\polar{\partial_c g}(x,y) = c \cdot (2x_2y_2,\, 2x_1y_2 + 2x_2y_1)$.
    On the other hand, $\polar{g}(x,y,z) = 2x_1 y_2 z_2 + 2y_1 x_2z_2 + 2z_1 x_2y_2$,
    so that $\partial_c \polar{g}(x,y) = \polar{g}(x,y,c)
    = c \cdot (2x_2y_2,\, 2x_1y_2+2x_2y_1)$ as well.
\end{example}

\begin{example}[Example for Proposition~\ref{prop:dd}]
For $d=2$ and $g(x)=x_1x_2x_3$, 
we have by definition 
\begin{align*}
    \D^2_{u,v} g(x) &= g(x+u+v)-g(x+u)-g(x+v)+g(x)\\
    &= (x_1+u_1+v_1)(x_2+u_2+v_2)(x_3+u_3+v_3)\\
    &- (x_1+u_1)(x_2+u_2)(x_3+u_3) - (x_1+v_1)(x_2+v_2)(x_3+v_3) + x_1x_2x_3.
\end{align*}
One can group the terms above in the following suggestive way:
\begin{align*}
    \D^2_{u,v} g(x) &= x_1(u_2v_3 + v_2u_3) + x_2(u_1v_3 + v_1u_3) + x_3(u_1v_2 + v_1u_2)\\
    &+ (u_1u_2v_3 + u_1v_2u_3 + v_1u_2u_3) + (v_1v_2u_3 + v_1u_2v_3 + u_1v_2v_3),
\end{align*}
and a short inspection reveals 
\begin{align*}
    \D^2_{u,v} g(x) 
    &= \big(x+\frac12u + \frac12v\big)\cdot \big(u_2v_3 + v_2u_3,\, u_1v_3 + v_1u_3,\, u_1v_2+v_1u_2\big) .
\end{align*}
Note that by Lemma~\ref{lemma:Delta-Nabla}, we can compute $\Nabla \polar{g}$ as $\polar{\Nabla g}$. And indeed, 
\[\D^2_{u,v} g(x) = \big(x+\frac12u + \frac12v\big)\cdot 
    \polar{\big(x_2x_3,\, x_1x_3,\, x_1x_2\big)}(u,v) \;.\]
\end{example}

\subsection{Discrete derivatives: Proofs}

\newcommand{\sur}{\twoheadrightarrow}

We first give a combinatorial description of the iterated discrete derivatives of any form (and by linearity, of any polynomial) in terms of \emph{semi-surjections}.
Henceforth, 
\[\varphi \colon [k]\sur[d] \text{ means } \varphi \colon [k]\to[d+1] \text{ with } \Image(\varphi) \supseteq [d].\]
We refer to a function $\varphi \colon [k]\sur[d]$ as a semi-surjection since it is surjective on its codomain except, perhaps, the last element, $d+1$.
Denote by $S_n$ the set of permutations $\sigma \colon [n] \to [n]$.

\begin{lemma}\label{lemma:add-deriv}
    For any form $f(x) = \sum_{I} c_I \prod_{j=1}^k x_{I_j}$, 
    with each $I=(I_1,\ldots,I_k)$ a sorted $k$-tuple of variables and $c_I$ a scalar,
    \[\D^d_{x^{(1)},\ldots,x^{(d)}} f(x^{(d+1)})
    = \sum_I c_I \sum_{\varphi \colon [k]\sur[d]}\, \prod_{j=1}^k x_{I_j}^{(\varphi(j))} .\] 
\end{lemma}
\begin{proof}
    By linearity we may assume that $f$ is a single monomial, $f = \prod_{j=1}^k x_{i_j}$,
    so it suffices to prove
    \[\D^d_{x^{(1)},\ldots,x^{(d)}} f(x^{(d+1)}) = \sum_{\varphi \colon [k]\sur[d]}\, \prod_{j=1}^k x_{i_j}^{(\varphi(j))}.\]
    We proceed by induction on $d \ge 0$. The base case $d=0$ holds
    since the only function $\varphi \colon [k]\to\{1\}$ is $\varphi=1$, 
    so the right hand side is 
    $\prod_{j=1}^k x_{i_j}^{(1)} = f(x^{(1)}) = \D^0_{\vec{x}} f(x^{(1)})$.

    For the induction step, abbreviate $\vec{x}=(x^{(1)},\ldots,x^{(d-1)})$, and observe that
    \begin{align*}
        \D^d_{x^{(1)},\ldots,x^{(d)}} f(x^{(d+1)})
        &= \D_{x^{(d)}} \D^{d-1}_{\vec{x}} f(x^{(d+1)})
        = \D^{d-1}_{\vec{x}} f(x^{(d+1)}+x^{(d)})
        - \D^{d-1}_{\vec{x}} f(x^{(d+1)})\\
        &=  \sum_{\varphi \colon [k]\sur[d-1]} \,
        \prod_{j\colon\varphi(j) \neq d} x_{i_j}^{(\varphi(j))} 
        \Big( \prod_{j\colon\varphi(j) = d} (x_{i_j}^{(d+1)} + x_{i_j}^{(d)})
        - \prod_{j\colon\varphi(j) = d} x_{i_j}^{(d+1)} \Big)\\
        &= \sum_{\varphi \colon [k]\sur[d-1]} \,
        \prod_{j\colon\varphi(j) \neq d} x_{i_j}^{(\varphi(j))} 
        \sum_{\substack{\tau \colon \inv{\varphi}(d)\to\{d,d+1\}\colon\\d \in \Image(\tau)}} \, 
        \prod_{j\colon\varphi(j) = d} x_{i_j}^{(\tau(j))}\\
        &= \sum_{\varphi \colon [k]\sur[d]} \, 
        \prod_{j=1}^k x_{i_j}^{(\varphi(j))} ,
    \end{align*}
    where the third equality uses the induction hypothesis twice,
    and the last equality uses the natural bijection between semi-surjections $\varphi' \colon [k]\sur[d]$
    and pairs 
    $(\varphi,\tau)$
    where $\varphi \colon [k]\sur[d-1]$ and 
    $\tau \colon \inv{\varphi}(d)\to\{d,d+1\}$ with $d \in \Image(\tau)$.
    This completes the induction step and the proof.
    
\end{proof}

We record two facts about polarizations, which can be deduced from Lemma~\ref{lemma:add-deriv} effortlessly.
\begin{fact}\label{fact:polar}
    For any monomial $g(x) = \prod_{j=1}^k x_{i_j}$,
    \[\polar{g}(x^{(1)},\ldots,x^{(k)}) 
    = \sum_{\sigma \in S_k}\, \prod_{j=1}^k x_{i_j}^{(\sigma(j))} \,.\]
\end{fact}

\begin{fact}\label{fact:diagonal}
    Any $g \in \Forms_k$ satisfies
    $\polar{g}(x,x,\ldots,x) = k!g(x)$.
\end{fact}

Using the explicit description of polarization, we now prove 
that polarization commutes with the formal derivative.
\begin{manuallemma}[\ref{lemma:Delta-Nabla}]
     For any form $g$ and direction vector $c$,
    \[\partial_c \polar{g} = \polar{\partial_c g}.\]
\end{manuallemma}
\begin{proof}[Proof of Lemma~\ref{lemma:Delta-Nabla}]
    It suffices to show
    $\partial_{x_n^{(k)}} \polar{g}(\x) = \polar{\partial_{x_n} g}(\x)$
    for $g(x) = \prod_{j=1}^{k} x_{i_j}$ ($i_1\le\cdots\le i_k$) a single monomial, and $\x=(x^{(1)},\ldots,x^{(k-1)})$,
    since the result would follow by symmetry and linearity (the polarization $\overline{g}$ is symmetric by Fact~\ref{fact:polar}; directional derivatives are linear in $c$).
    On the one hand,\footnote{For example, $\partial_{x_2}(x_1x_2^2) = x_1x_2 + x_1x_2$.}
    \[\partial_{x_n} g(x) 
    = \sum_{\substack{t \in [k]\\i_t = n}}\, \prod_{\substack{j=1\\j\neq t}}^k x_{i_j}\,
    = \sum_{\substack{t \in [k]\\i_t = n}}\, \prod_{j=1}^{k-1} x_{i_j} ,\]
    and therefore
    \begin{equation}\label{eq:derive1}
        \polar{\partial_{x_n} g}(x^{(1)},\ldots,x^{(k-1)}) 
    = \sum_{\sigma \in S_{k-1}}\, \sum_{\substack{t \in [k]\\i_t = n}}\, \prod_{j=1}^{k-1} x_{i_j}^{(\sigma(j))}.
    \end{equation}
    On the other hand,
    \[\polar{g}(x^{(1)},\ldots,x^{(k)}) = \sum_{\sigma \in S_k} \, \prod_{j=1}^{k} x_{i_j}^{(\sigma(j))}\]
    is multilinear, so its only monomials containing $x^{(k)}_n$ are those where an occurrence of $x_n$ is replaced by $x_n^{(k)}$ (that is, $i_t=n$ and $\sigma(t)=k$); therefore,
    \begin{equation}\label{eq:derive2}
    \partial_{x^{(k)}_n} \polar{g}(x^{(1)},\ldots,x^{(k-1)}) 
    = \sum_{\substack{t \in [k]\\i_t = n}}\, \sum_{\substack{\sigma \in S_k\\\sigma(t)=k}}\, \prod_{\substack{j=1\\j \neq t}}^k x_{i_j}^{(\sigma(j))}
    = \sum_{\substack{t \in [k]\\i_t = n}}\, \sum_{\substack{\sigma \in S_k\\\sigma(t)=k}}\, \prod_{j=1}^{k-1} x_{i_j}^{(\sigma(j))} \;.
    \end{equation}
    A quick inspection shows that~(\ref{eq:derive1}) and~(\ref{eq:derive2}) are equal, as needed.  

\end{proof}

We are now ready to prove our formula for the order-$d$ derivative of forms of degree $d+1$.

\begin{proof}[Proof of Proposition~\ref{prop:dd}]
    By linearity, we may assume that we are given a single monomial, $f = \prod_{j=1}^{d+1} x_{i_j}$, of degree $d+1$.
    We need to show, for $\vec{x}=(x^{(1)},\ldots,x^{(d)})$, that 
    \[\D^d_{\vec{x}} f(x^{(d+1)}) = \polar{f}\Big(\vec{x},\, x^{(d+1)} + \frac12\sum_{t=1}^d  x^{(t)} \Big).\]
    For a function $\varphi \colon [d+1]\to[d+1]$, we abbreviate $f^{(\varphi)} = \prod_{j=1}^{d+1} x_{i_j}^{(\varphi(j))}$.
    Observe that every semi-surjection $\varphi \colon [d+1]\sur[d]$ is either a permutation or takes exactly one value twice.
    Therefore, by Lemma~\ref{lemma:add-deriv},
    \begin{align*}
    \D^d_{\vec{x}} f(x^{(d+1)}) 
    &= \sum_{\varphi \colon [d+1]\sur[d]}\, f^{(\varphi)}\\
    &= \sum_{\sigma \in S_{d+1}}\, f^{(\sigma)}\,
    +\, \sum_{t=1}^d\, \sum_{\varphi \in S_d^t}\, f^{(\varphi)}
    \end{align*}
    where $S_d^t$ denotes the set of all surjections $\varphi \colon [d+1]\to[d]$ satisfying $|\inv{\varphi}(t)|=2$.
    For $t \in [d]$ and a permutation $\sigma \in S_{d+1}$, 
    consider the function $\sigma_t \in S_d^t$ given by
    \[\sigma_t(x) := \begin{cases}
        t           & \text{if }\sigma(x) = d+1\\
        \sigma(x)   & \text{else}
    \end{cases}\;.\]
    We claim that the map $\sigma \mapsto \sigma_t$, from $S_{d+1}$ to $S_d^t$, is $2$-to-$1$.
    Indeed, for any $\varphi \in S_d^t$, write $\varphi^{-1}(t)=\{a,b\}$ and consider the two permutations
    \[\varphi_{a \mapsto d+1}(x)
    = \begin{cases}
        d+1    & \text{if }x = a\\
        \varphi(x)   & \text{else}
    \end{cases}
    \quad\text{ and }\quad
    \varphi_{b \mapsto d+1}(x)
    = \begin{cases}
        d+1    & \text{if }x = b\\
        \varphi(x)   & \text{else}
    \end{cases}.\]
    A moment's thought reveals that 
    $\sigma \in S_{d+1}$ satisfies $\sigma_t=\varphi$ if and only if $\sigma \in \{\varphi_{a \mapsto d+1},\varphi_{b \mapsto d+1}\}$.

    

    Going back to $\D^d_{\vec{x}} f(x^{(d+1)})$, 
    and using the fact that $\ch(\F) \neq 2$,
    we thus have
    \[\sum_{\varphi \in S_d^t}\, f^{(\varphi)}
    = \frac12\sum_{\sigma \in S_{d+1}} f^{(\sigma_t)}.\]
    Therefore,
    \begin{align*}
        \D^d_{\vec{x}} f(x^{(d+1)}) &= 
        \sum_{\sigma \in S_{d+1}}\, f^{(\sigma)}\,
        +\, \sum_{t=1}^d\, \frac12\sum_{\sigma \in S_{d+1}} f^{(\sigma_t)}\\
        &= \polar{f}(\vec{x},x^{(d+1)}) + \sum_{t=1}^d\, \frac12\polar{f}(\vec{x},x^{(t)}) \\
        &= \polar{f}\Big(\vec{x}, x^{(d+1)} + \frac12\sum_{t=1}^d  x^{(t)} \Big).
    \end{align*}
    using Fact~\ref{fact:polar} in the second equality, and the multilinearity of $\polar{f}$ in the last equality. This completes the proof.
\end{proof}

\section{Proof of Theorem \ref{thm:d+1-PGI-poly}}\label{sec:main-pf}

\subsection{Bias of general polynomials}

While the bias of a multilinear form is roughly characterized by its rank (recall Theorem \ref{thm:SvR}), the situation is more involved for general, non-homogeneous polynomials.

\begin{remark}
    While the bias of a polynomial is closely related to the bias of its top-degree component, it can have  
    little to do with the bias of its lower-degree components.
    For example, consider the univariate polynomial $x^2+x$, over $\F_p$ with $p>2$.
    We have
    \[\bias(x^2+x)\cdot p = \Big|\sum_{x \in \F_p} e(x^2+x)\Big|
    = \Big|e(-1/4)\sum_{x \in \F_p} e((x+1/2)^2)\Big|
    = \Big|\sum_{y \in \F_p} e(y^2)\Big|
    = \sqrt{p},\]    
    where the last step uses the formula for a basic quadratic Gauss sum.
    Thus, $\bias(x^2+x)=\bias(x^2) = p^{-1/2}$. 
    However, for the lower degree component of $x^2+x$ we have $\bias(x) = 0$.
\end{remark}

\begin{remark}\label{remark:D-bias-nonhom}
    For non-homogeneous polynomials $f$,
    having large $\bias(\D^d f)$ does not guarantee that the lower-degree components of $f$ are structured or biased. 
    For example, consider $f(x,y)=h(x)\cdot \ell(y) \in \Poly_{d+1}(\F_q)$, where $\ell$ is non-homogeneous with $\deg(\ell)=1$
    and $h \in \Forms_d$ is a random $n$-variate form of degree $d \ge 2$.
    The degree-$d$ component of $f$ is a random homogeneous polynomial
    (in particular, $\bias(f_{d}) \le q^{-\Omega(n)}$).
    However, $\bias(f) = \P(h=0) \approx 1/q$ by a quick calculation,\footnote{$\bi(f) = \Exp_x \Exp_y \chi(h(x)l(y)) 
    =\Exp_x \mathbf{1}(h(x)=0) = \P(h=0)$.}
    so $\bias(\D^d f) \ge \bias(f)^{2^d} \ge q^{-O_d(1)}$ is large as $n\to\infty$.
\end{remark}

First, we prove a rather general chain rule for bias.
For an event $E$, we henceforth define the \emph{conditional bias} as $\bi(f \vvert E) = \Exp(\chi(f) \vvert E)$.

\begin{lemma}[Chain rule for bias, a general form]\label{lemma:bias-chain}
    Let $A(x,y)$ and $B(y)$ be functions over a finite field, and assume $A$ is linear in $x$.
    For the event $E=\{y \vert \forall x \colon A(x,y)=0\}$, 
    we have
    \[\bi(A+B) = \bi(A)\bi(B \vvert E).\]
\end{lemma}
\begin{proof}
    Observe that for any fixed $y \in E^c$, $A(x,y)+B(y)$ 
    has degree exactly $1$ (in $x$),
    and thus 
    takes any value in the field with the same probability.
    Furthermore observe that
    \[\bi(A) = \Exp(\chi(A)) = \Exp_y \Exp_x(\chi(A(x,y))) = \Exp_y \1_E(y) = \P(E).\]
    Therefore, by the law of total expectation,
    \begin{align*}
        \bi(A+B) &= \Exp(\chi(A+B)) 
        = \P(E)\Exp(\chi(A+B) \vvert E) + \P(E^c)\Exp(\chi(A+B) \vvert E^c)\\
        &= \bi(A)\bi(B \vvert E) + \P(E^c)\cdot 0. \qedhere
    \end{align*}
\end{proof}

Using Proposition~\ref{prop:dd} and Lemma~\ref{lemma:bias-chain}, our main result in this subsection
bounds the bias of the $d$-times iterated derivative of $f \in \Poly_{d+1}$ in terms of (polarizations of) the components of $f$.
\begin{proposition}\label{prop:derive-bias}
    For $f \in \Poly_{d+1}(\F)$ where $\ch(\F)> 2$,
    with $g=f_{d+1}$ and $h=f_d$,
    \[\bias(\D^d f) \le \Exp_c \bias(\polar{h} - \partial_c\polar{g})
    \quad\text{ and }\quad 
    \bias(\D^d f) \le \bias(\polar{g}) .\]
\end{proposition}

\begin{proof}
    By Proposition~\ref{prop:dd}, 
    $\D^d_{\vec{v}} g(x) = \polar{g}(\vec{v}, x')$
    where $x' = x+\frac12\sum_{t=1}^d v^{(t)}$.
    Thus, 
    $\D^d_{\vec{v}} f(x) = \polar{g}(\vec{v},x') + \polar{h}(\vec{v})$,
    and therefore
    \[\bi(\D^d f) = \bi(\D^d_{\vec{v}} f(x)) 
    = \Exp_x \Exp_\vec{v} \chi(\polar{g}(\vec{v},x') + \polar{h}(\vec{v}))
    = \Exp_{x'} \Exp_\vec{v} \chi(\polar{g}(\vec{v}, x') + \polar{h}(\vec{v}))
    = \bi(\polar{g} + \polar{h}).\]
    This has two implications.
    First, it follows from Lemma~\ref{lemma:bias-chain}
    that
    \[\bias(\polar{g} + \polar{h}) \le \bias(\polar{g}).\]
    Second, we have
    \begin{align*}
        \bi(\polar{g} + \polar{h}) 
        &= \Exp_x \Exp_{\vec{v}} \chi(\polar{g}(\vec{v},x) + \polar{h}(\vec{v}))
        = \Exp_x \Exp_{\vec{v}} \chi(\partial_x \polar{g}(\vec{v}) + \polar{h}(\vec{v}))\\
        &= \Exp_c \Exp_{\vec{v}} \chi(\partial_{-c}\, \polar{g}(\vec{v}) + \polar{h}(\vec{v}))
        = \Exp_c \bi(\polar{h}-\partial_{c} \polar{g}).
    \end{align*}
    Taking absolute value and using the triangle inequality completes the proof.
\end{proof}

\subsection{Main proof}

We now put everything together---in particular, Propositions~\ref{prop:derive-bias}, \ref{prop:rk*-inh}, \ref{prop:corr} and Theorem~\ref{thm:SvR}---to prove Theorem \ref{thm:d+1-PGI-poly}.

\begin{manualthm}[\ref{thm:d+1-PGI-poly}]
    For every $f \in \Poly_{d+1}(\F_q)$, with $\ch(\F)>d+1$,
    \[\cor_{<d}(f) \ge \bias(\D^d f)^{\Ot_d(1)}.\]
\end{manualthm}
\begin{proof}
    For $d=1$ there is nothing to prove, since $\cor_{<1}(f)^2 = \bias(\D f)$;
    indeed, $\bias(f)^2 = \bias(\D f) \ge \cor_{<1}(f)^2 \ge \cor(f,0)^2 = \bias(f)^2$ using~(\ref{eq:GI-easy}).
    We thus assume $d \ge 2$.\footnote{Recall that polynomial bounds for the Gowers inverse theorem are known  also for $d=2$ (via Fourier analysis; e.g., Proposition~2.2 in~\cite{GT08}) and $d=3$~\cite{GGMT24,GGMT25}. However, our proof does not rely on any of these.}
    
    Let $g=f_{d+1}$ and $h=f_d$
    be the components of $f$ of degrees $d+1$ and $d$, respectively.
    We first claim that, for any direction vector $c$,
    \begin{equation}\label{eq:rk-bounds}
        \rk(h - \partial_c g) \le \rk(\polar{h} - \partial_c\polar{g}) \quad\text{ and }\quad 
        \rk(g) \le \rk(\polar{g}).
    \end{equation}
    To see this, abbreviate $\x^t=(x,\ldots,x)$ ($t$ times), and observe that
    $\partial_c \polar{g}(\x^d) = \polar{\partial_c g}(\x^d) = d!\partial_c g(x)$,
    using Lemma~\ref{lemma:Delta-Nabla} and Fact~\ref{fact:diagonal}.
    Moreover by Fact~\ref{fact:diagonal},
    we have 
    $\polar{h}(\x^d) = d!h(x)$
    and $\polar{g}(\x^{d+1}) = (d+1)!g(x)$.
    As $\ch(\F) > d+1$,
    we thus have $g(x) = \polar{g}(\x^{d+1})/(d+1)!$ and 
    $(h - \partial_c g)(x) = (\polar{h} - \partial_c\polar{g})(\x^d)/d!$, 
    proving~(\ref{eq:rk-bounds}) (since plugging in the diagonal in a reducible multilinear form gives a reducible form).
    
    By Proposition~\ref{prop:derive-bias},
    $\bias(\D^d f) \le \bias(\polar{g})$,
    and 
    $\bias(\D^d f) \le \bias(\polar{h} - \partial_c\polar{g})$
    for some direction vector $c$ which we henceforth fix.
    Put $\bias(\D^d f) = q^{-r}$ (if $\bias(\D^d f)=0$ there is nothing to prove),
    so that
    \begin{equation}\label{eq:bias-bounds}
        \bias(\polar{g}) \ge q^{-r} \quad\text{ and }\quad 
        \bias(\polar{h} - \partial_c\polar{g}) \ge q^{-r}.
    \end{equation}
    Now, combine~(\ref{eq:rk-bounds}), (\ref{eq:bias-bounds}), and 
    Theorem~\ref{thm:SvR} (using $\rk(T) \le \PR(T)$)
    to deduce that
    \[
        \rk(g) \le \rk(\polar{g}) \le \Ot_d(r) \quad\text{ and }\quad 
        \rk(h - \partial_c g) \le \rk(\polar{h} - \partial_c\polar{g}) \le \Ot_d(r).
    \]
    Apply Proposition~\ref{prop:rk*-inh} 
    to obtain
    \[\rk^*(g+h+\g) \le \rk(g) + \rk(h-\partial_c g) \le \Ot_d(r)\]
    for some $\g \in \Poly_{d-1}$.
    Next, apply Proposition~\ref{prop:corr} with $k=d+1$ ($\ge 3$ by assumption) on $g+h+\g \in \Poly_{d+1}$.
    We deduce that there exists a polynomial $P$ with 
    $\deg(P) \le k-2 = d-1$ such that
    \[\cor(g+h+\g,\,P) \ge q^{-2\rk^*(g+h+\g)} \ge q^{-\Ot_d(r)}.\]
    To finish, write $f=g+h+f_{<d}$.
    Note that $P+f_{<d}-\g \in \Poly_{d-1}$; 
    we claim that it
    significantly correlates with $f$. Indeed,
    \begin{align*}
        \cor_{<d}(f)
        &\ge \cor(f,\,P+f_{<d}-\g) = \bias(f-(P+f_{<d}-\g)) 
        = \bias(g+h+\g-P)\\ 
        & = \cor(g+h+\g,\,P) \ge q^{-\Ot_d(r)} 
        = \bias(\D^d f)^{\Ot_d(1)}.
    \end{align*}
    This completes the proof.
\end{proof}

\paragraph*{Acknowledgment.} We thank Daniel G.~Zhu for insightful comments on an earlier draft.

\end{document}